\documentclass{amsart}
\usepackage{tikzit}

\tikzstyle{red}=[fill={rgb,255: red,191; green,0; blue,64}, draw=black, shape=circle]
\tikzstyle{blue}=[fill=blue, draw=black, shape=circle]
\tikzstyle{green}=[fill=green, draw=black, shape=circle]
\tikzstyle{black}=[fill=black, draw=black, shape=circle]
\tikzstyle{white}=[fill=white, draw=black, shape=circle]

\tikzstyle{dir}=[<-]
\tikzstyle{dir}=[->]
\tikzstyle{dashed }=[-, dashed]

\usepackage{graphicx}
\usepackage{float}
\usepackage{amssymb}
\usepackage{epstopdf}
\usepackage{amsmath,amsthm,amscd,amssymb}
\usepackage{latexsym}
\usepackage[colorlinks,citecolor=red,pagebackref,hypertexnames=false]{hyperref}
\usepackage{geometry}                
\usepackage[us, 12hr]{datetime}

\numberwithin{equation}{section}

\theoremstyle{plain}
\newtheorem{theorem}{Theorem}[section]
\newtheorem{lemma}[theorem]{Lemma}

\theoremstyle{definition}
\newtheorem{definition}[theorem]{Definition}

\newtheorem{problem}[theorem]{Problem}

\newtheorem{case[theorem]}{Case}

\theoremstyle{remark}
\newtheorem{remark}[theorem]{Remark}

\numberwithin{equation}{section}

\begin{document}
\title{On complete and incomplete exponential systems} 


\author{Alex Iosevich and Azita Mayeli}

\date{Last edit: \currenttime, \today}

\thanks{The work of the first listed author was supported in part by the National Science Foundation under grant no. HDR TRIPODS - 1934962}

\begin{abstract} Given a bounded domain $\Omega \subset {\Bbb R}^d$ with positive measure and a finite set $A=\{a^1, a^2, \dots, a^d\}$, we say that the set 
${\mathcal E}(A)={\{e^{2 \pi i x \cdot a^j}\}}_{a^j \in A}$ is a complete exponential system if for every $\xi \in {\Bbb R}^d$, there exists $1 \leq j \leq d+1$ such that 
\begin{equation} \label{completedef} \int_{\Omega} e^{-2 \pi i x \cdot (a^j-\xi)} dx \not=0; \end{equation} otherwise ${\mathcal E}(A)$ is called an incomplete exponential system. In this paper, we essentially classify complete and incomplete exponential systems when $\Omega=B_d$, the unit ball, and when $\Omega=Q_d$, the unit cube. 

Given a bounded domain $\Omega$, we say that $e^{2 \pi i x \cdot a}, e^{2 \pi i x \cdot a'}$ are $\phi$-approximately orthogonal if 
$$|\widehat{\chi}_{\Omega}(a-a')| \leq \phi(|a-a'|), \ a\neq a'$$ where $\phi: [0, \infty) \to [0, \infty)$ is a bounded measurable function that tends to $0$ at infinity. We prove that $L^2(B_d)$ does not possess a $\phi$-approximate orthogonal basis of exponentials for a wide range of functions $\phi$. The proof involves connections with the theory of distances in sets of positive Lebesgue upper density originally developed by Furstenberg, Katznelson and Weiss (\cite{FKW90}).  
\end{abstract} 

\maketitle

\section{Introduction} The study of exponential functions on domain in ${\Bbb R}^d$ is an old and time-honored subject. Let $A$ be a discrete subset of ${\Bbb R}^d$ and define 
$$ {\mathcal E}(A)=\{e^{2 \pi i x \cdot a}, \ a \in A\},$$ the set of exponentials with frequencies in $A$. Let $\Omega$ be a bounded domain in ${\Bbb R}^d$ of positive Lebesgue measure. 


\begin{definition} We say that ${\mathcal E}(A)$ is an orthogonal basis of $L^2(\Omega)$ if ${\mathcal E}(A)$ is an orthogonal exponential system over  $\Omega$ and also a basis for $L^2(\Omega)$. \end{definition} 

The study of orthogonal exponential bases in recent decades centered around the celebrated Fuglede Conjecture (\cite{Fug74}) which says that if $\Omega$ be a bounded domain in ${\Bbb R}^d$ of positive Lebesgue measure, then $L^2(\Omega)$ possesses an orthogonal basis of exponentials if and only if $\Omega$ tiles ${\Bbb R}^d$ by translation in the sense that there exists a discrete set $T \subset {\Bbb R}^d$ such that 
$$ \sum_{\tau \in T} \chi_{\Omega}(x-\tau)=1 \ a.e. \ x\in\Bbb R $$

While the Fuglede Conjecture is, in general, false, as shown by Terry Tao (\cite{T04}) in one direction, and by Kolountzakis-Matolcsi (\cite{KM06}) in the other, it has given rise to a multitude of fruitful investigations that have significantly improved our understanding of orthogonal exponential bases and tiling in Euclidean space and locally compact abelian groups. See, for example, \cite{GL17,GL20,IKT01,IKT03, IKP01,IMP17,LevM19,K99,M05,PhilippBirklbauer,FMV2019}, and the references contained therein. 
 
The Fuglede conjecture spawned a number of related questions that are interesting in their own right. For example, Fuglede, Iosevich, Kolountzakis, Rudnev, and others studied the following question. Given a bounded domain $\Omega \subset {\Bbb R}^d$ of positive Lebesgue measure, what is the largest possible size of $A \subset {\Bbb R}^d$ such that ${\mathcal E}(A)$ is an orthogonal exponential system in $L^2(\Omega)$? See, for example, \cite{Fug01}, \cite{IR03}, \cite{IK13}, and the references contained therein. In particular, these authors have shown that if $\Omega=B_d$, the unit ball in dimensions $\ge 2$ and ${\mathcal E}(A)$ is an orthogonal system with respect to $B_d$, then $A$ is finite. It is widely believed that the size of $A$ cannot exceed $d+1$, but no quantitative result of any sort is currently known. In \cite{IR03} the authors constructed a symmetric convex body $K$ with a smooth boundary and everywhere non-vanishing curvature in ${\Bbb R}^d$, $d \equiv 1 \mod 4$, such that there exists an infinite set $A$ where ${\mathcal E}(A)$ is an orthogonal system in $L^2(K)$. They also proved that if $d \not= 1 \mod 4$, then under the same assumptions $A$ must be finite. We are not aware of any results on bounds for the number of orthogonal exponentials on general bounded domains in ${\Bbb R}^d$. \\

A related question was formulated and considered by Lai and  the second listed author in \cite{LM17} in the special case of lattices and some partial results were obtained in that context, as we explain it below. 

\begin{definition} \label{expon_complete}(\cite{LM17}) Let $\Omega$ be a bounded domain in ${\Bbb R}^d$ of positive Lebesgue measure. Given a countable set $A\subset \Bbb R^d$, we say that ${\mathcal E}(A)={\{e^{2 \pi i x \cdot a}\}}_{a \in A}$ is a complete exponential system in $L^2(\Omega)$ if for any given $\xi \in {\Bbb R}^d$ there exists $a \in A$ such that 
$$ \int_{\Omega} e^{2 \pi i x \cdot (a-\xi)} dx \not=0.$$
Otherwise, we say the system ${\mathcal E}(A)$ is exponentially incomplete. 
\end{definition} 

Definition \ref{expon_complete} suggests the study of the following problem. 

\begin{problem}\label{mamaproblem} Given a bounded domain $\Omega\subset \Bbb R^d$ with positive measure, identify all discrete and countable  sets $A\subset \Bbb R^d$  for which the  set $\mathcal E(A)= \{e^{2\pi i x\cdot a} : \ a\in A\}$ is a complete exponential system in $L^2(\Omega)$. \end{problem}  


The exponential completeness problem is motivated in \cite{LM17} by the study of Gabor bases generated by the characteristic function of a bounded domain. Let $\Omega\subset \Bbb R^d$ be a bounded domain of positive Lebesgue measure and let $S\subset \Bbb R^{2d}$ be discrete and countable.  The Gabor system generated by the characteristic (window) function $g:=\chi_\Omega$ with respected to  $S$ is defined as
\begin{align}\label{Gabor_sys}
\mathcal G(g, S)=\left\{ e^{2\pi x\cdot b} \chi_\Omega(x-a): \quad (a,b)\in S\right\}.
\end{align}
If $G(g,S)$ is an orthogonal basis for $L^2(\Bbb R^d)$, we say it is a Gabor basis and $S$ is called a {\it Gabor spectrum}. Of particular interest  is the set $S$ where the entries in each pair    $(a,b)$ are correlated vectors.  For  example, if  $S$ is a lattice, in many cases it can be identified as  $S=M(\Bbb Z^{2d})$, where $M=[I, 0: C, I]$ is a $2d\times 2d$  upper triangular block matrix  and $I$ is the  $d\times d$ identity matrix. In this case, every pair in  $S$ has the form $(m,Cm+n)$, for some $m,n\in \Bbb Z^d$,  and with the second entry  related to the first entry through the matrix $C$. In this situation, the orthogonally of  functions in the Gabor system (\ref{Gabor_sys})  holds if 
 
\begin{align}\label{orthog}
\int_{\Omega\cap \Omega+m} e^{-2\pi i Cm\cdot x} e^{-2\pi i n\cdot x} dx =0 \quad \forall  m, n \in {\Bbb Z}^d .
\end{align}

The equality  (\ref{orthog}) shows that the orthogonality of the Gabor system  (\ref{Gabor_sys})  in this example is equivalent to say that for a fixed   $m\in {\Bbb Z}^d$,  the exponential function $e^{-2 \pi i Cm \cdot x}$ is orthogonal to the exponential system $\mathcal E(\Bbb Z^d)=\{e^{2 \pi i n \cdot x}: n\in {\Bbb Z}^d\}$ with respect to  the domain $\Omega \cap \Omega+m$. \\
 
The goal of this paper is to study Problem \ref{mamaproblem} in the case when $\Omega$ is the unit cube or the unit ball in $\Bbb R^d$,  and the set of exponentials $\mathcal E(A)$  is finite. Our methods  {\it  probably  extend} to  the case when $\Omega$ is a bounded symmetric convex set with a smooth boundary and everywhere non-vanishing Gaussian curvature. The obstacle to study this problem in more generality is the lack of detailed knowledge of the behavior of the zero set of the Fourier transform of the indicator function of a general bounded domain. 

\vskip.125in 

Our first result is the following. 

\begin{theorem} \label{maingeneral} Let $\Omega=B_d$, $d \ge 2$, the unit ball, or $Q_d$, $d \ge 2$, the unit cube. Then the following hold. 

\vskip.125in 

i) The $d(d+1)$-dimensional Lebesgue measure of the $(d+1)$-tuples $(a^1,a^2, \dots, a^{d+1}) \in {({\Bbb R}^d)}^{d+1}$ 

such that ${\mathcal E}(A)$, $A=\{a^1,a^2, \dots, a^{d+1}\}$, is not a complete exponential system in $L^2(\Omega)$ is $0$. 

\vskip.125in 

ii) Suppose that $\Omega=Q_d$, the unit cube and $A\subset \Bbb R^d$.  Then if $\# A \leq d$, then ${\mathcal E}(A)$ is not a

 complete exponential system in $L^2(Q_d)$. 

\vskip.125in 

iii) If $\Omega=B_d$, the unit ball,  then the $d^2$-dimensional Lebesgue measure of the set of $d$-tuples 

$(a^1, \dots, a^d)$ such that ${\mathcal E}(A)$ is a complete exponential system in $L^2(B_d)$ is $0$. 

\vskip.125in 

iv) In both the case of the cube $\Omega=Q_d$ and $\Omega=B_d$, for any $N \ge d+1$ there exist $a^1, \dots, a^{N}$ 

in ${\Bbb R}^d$ such that ${\mathcal E}(A)$ is not a complete system in $L^2(\Omega)$. 

\vskip.125in 

v) If $\Omega=B_2,$ the unit ball in ${\Bbb R}^2$, and if $A=\{a^1,a^2\}$, $a^1 \not=a^2$, then ${\mathcal E}(A)$ is always incomplete. 

\vskip.125in 

vi) If $\Omega=B_d$, the unit ball, and $d \ge 3$, then the set of $d$-tuples  $(a^1, \dots, a^d) \in {({\Bbb R}^d)}^{d}$ such that

 ${\mathcal E}(A)$ is a complete exponential system in $L^2(B_d)$ is not empty. In other words, there existdata 
 
 $A=\{a^1, \dots, a^d\}$, $a^j \in {\Bbb R}^d$ such that ${\mathcal E}(A)$ is a complete exponential system in $L^2(B_d)$. 

\end{theorem} 

\vskip.125in 

\begin{remark} If $d=1$, the cube is the same as the ball and the questions raised above are more or less straightforward. Since the Fourier transform of the indicator function of $[0,1]$ vanishes on ${\Bbb Z}$, given any $a \in {\Bbb R}$, we can choose $a' \in {\Bbb R}$ with $a-a' \in {\Bbb Z}$, and this makes $e^{2\pi i x \cdot a}$ and $e^{2 \pi i x \cdot a'}$ orthogonal in $L^2([0,1])$. Given any two vectors $a,a' \in {\Bbb R}$, the question of whether there exists $b \in {\Bbb R}$ such that $e^{2 \pi i x \cdot b}$ is orthogonal to both $e^{2\pi i x \cdot a}$ and $e^{2\pi i x \cdot a'}$ in $L^2([0,1])$ comes down to whether $b-a$ and $b-a'$ are both integers. If $b-a=k$ for some integer $k$, and $b-a'=k'$ for some integer $k'$, then $k+a=k'+a'$, which forces $a-a'$ to be an integer. As the reader shall see, the situation in higher dimensions is more interesting. \end{remark} 

\begin{remark} The statement of Theorem \ref{maingeneral} should hold, with small adjustments, if $\Omega$ is any symmetric convex body in ${\Bbb R}^d$, $d \ge 2$. In the case when $\Omega$ is a symmetric convex body with a smooth boundary and everywhere non-vanishing Gaussian curvature, this can probably be accomplished using an elaboration on the techniques of this paper, in view of the results on the zero set of the Fourier transform (see e.g. \cite{Z84}). The case of the general symmetric convex body is probably less accessible as the zero set of the Fourier transform of its characteristic function is more difficult to describe. \end{remark}

\vskip.125in 

\subsection{$\phi$-approximate orthogonality} We now loosen the notion of orthogonality a bit to expose some salient geometrical features of  the underlying domains. 

\vskip.125in 

\begin{definition} Given a bounded domain $\Omega\subset \Bbb R^d$,  we say that for $a\neq a'$ the exponentials  $e^{2 \pi i x \cdot a}, e^{2 \pi i x \cdot a'}$ are $\phi$-approximately orthogonal if 
$$|\widehat{\chi}_{\Omega}(a-a')| \leq \phi(|a-a'|), $$
 where $\phi: [0, \infty) \to [0, \infty)$ is a bounded measurable function that tends to $0$ at infinity. \end{definition} 
It is clear that orthogonal exponential functions on $\Omega$  are $\phi$-approximately orthogonal if we choose $\phi=0$. 

\begin{definition} Given a bounded domain $\Omega$, we say that ${\mathcal E}(A)$ is a $\phi$-approximately orthogonal basis for $L^2(\Omega)$ if ${\mathcal E}(A)$ is a basis for $L^2(\Omega)$ and any two distinct elements of ${\mathcal E}(A)$ are $\phi$-approximately orthogonal. \end{definition} 

In this paper, we shall primarily focus on the $\phi$-orthogonality in the case $\Omega=B_d$, but we plan to engage in a more systematic study in the sequel. It is well-known (see \ref{ballfourierformula} below) that 
\begin{align}\label{upper-esitmation}
 |\widehat{\chi}_{B_d}(\xi)| \leq C{(1+|\xi|)}^{-\frac{d+1}{2}},
 \end{align}
  which implies that if 
$$\phi(t)=c{(1+t)}^{-\frac{d+1}{2}}$$ with a suitable constant $c$, then every collection ${\mathcal E}(A)$ is $\phi$-approximately orthogonal. It follows that the only non-trivial case of $\phi$-orthogonality in $L^2(B_d)$ is when the decay rate of $\phi$ as $t \to \infty$ is faster than the uniform radial decay rate of $\widehat{\chi}_{B_d}$. Our main result in this direction is the following. 

\begin{theorem} \label{mainapprox} Let $\phi: [0, \infty)\to [0, \infty)$ be a any bounded and measurable  function such that 
\begin{equation} \label{phismall} \lim_{t \to \infty} {(1+t)}^{\frac{d+1}{2}} \phi(t)=0. \end{equation} 

Then there does not exist a set $A \subset {\Bbb R}^d$, $d>1$,  such that $L^2(B_d)$ possesses a $\phi$-approximate orthogonal basis ${\mathcal E}(A)$. 
\end{theorem} 

\vskip.125in 

\begin{remark} The conclusion of Theorem \ref{mainapprox} is false in ${\Bbb R}$ because $L^2([0,1])$ has an orthogonal basis of exponentials. \end{remark} 

\vskip.125in 

\begin{remark} A variety of questions related to the notion of $\phi$-approximate orthogonality can and should be addressed in the sequel. For example, Theorem \ref{maingeneral} can be reexamined with exponential completeness replaced by the notion of $\phi$-approximate exponential completeness. Another area worthy of attention is a detailed study of the structure of $\phi$-orthogonal families in $L^2(\Omega)$, where $\Omega$ is a given bounded domain in ${\Bbb R}^d$. In \cite{IR03}, the first listed author and Rudnev proved that if ${\mathcal E}(A)$ is orthogonal with respect to $L^2(K)$, where $K$ is a bounded symmetric convex set with a smooth boundary and everywhere non-vanishing curvature, then $A$ is finite if $d \not=1 \mod 4$, otherwise $A$ is finite or contained in a  line in ${\Bbb R}^d$. The proof of that result shows that it still holds if orthogonality is replaced by $\phi$-orthogonality for $\phi$'s that are sufficiently rapidly decaying at infinity, and an interested reader can use the mechanism in \cite{IR03} and this paper to obtain a fully quantitative version of this statement. A detailed study of this phenomenon would be quite interesting.  
\end{remark}

\section{Proof of Theorem \ref{maingeneral}}

\vskip.125in 

\subsection{Proof of part ii)} Let $a^1, a^2, \dots, a^d$ denote a collection of vectors in ${\Bbb R}^d$, where 
$$a^j=(a^j_1, a^j_2, \dots, a^j_d).$$ It is not difficult to see that there exists $b\in \Bbb R^d$ such that $e^{2 \pi i x \cdot b}$ is orthogonal to every $e^{2 \pi i x \cdot a^j}$, $j=1,2, \dots, d$ in $L^2(Q_d)$. Indeed, take 
$$b=(a^1_1+1, a^2_2+1, \dots, a^d_d+1).$$ 

Then 
\begin{equation} \label{cubefourier} \int_{Q_d} e^{2 \pi i x \cdot (b-a^j)} dx=\prod_{k=1}^d \int_0^1 e^{2 \pi i x_k(b_k-a^j_k)} dx_k, \end{equation} and $j$'th element of the product on the right hand side is equal to $0$ since $b$ differs from $a^j$ by $1$ in the $j$'th coordinate. This establishes part ii) for $Q_d$ since if the number of exponentials is $<d$, the same argument works. 

\vskip.125in 

\subsection{Proof of part iii)} Suppose that $A=\{a^1, a^2, \dots, a^d\}$ and suppose that the vectors 
$$\{a^2-a^1, a^3-a^1, \dots, a^d-a^{1}\}$$ are linearly independent. Recall that (\cite{H62}) when $d>1$ 
\begin{equation} \label{ballfourierformula} \widehat{\chi}_{B_d}(\xi)={|\xi|}^{-\frac{d}{2}} J_{\frac{d}{2}}(2 \pi |\xi|), \end{equation}  and that the zeroes of the Bessel function $J_{\frac{d}{2}}$ are uniformly separated and are of the form 
\begin{equation} \label{zeroes} \frac{m}{2}+\frac{d-1}{8}+O \left(\frac{1}{m} \right), \ m \in {\Bbb Z}^{+}. \end{equation} 

By elementary geometry, there exists a line equidistant from each $a^j$. To see this, assume, without loss of generality that $a^1=(0, \dots, 0)$. The intersection of spheres of sufficiently large radius $r$ centered at $a^2, \dots, a^{d-1}$ consists of precisely two points by the linear independence condition, and as the radius varies, these points trace out a line. 

Choose a point on this line a distance $R$ from each $a^j$, such that $J_{\frac{d}{2}}(2 \pi R)=0$ and call this point $b$. Then it is clear that 
$$ \int_{B_d} e^{2 \pi i x \cdot (b-a^j)} dx=0$$ for each $j$ and the proof is complete since the set of $d$ tuples of vectors in ${\Bbb R}^d$ that do not satisfy the condition that the vectors 
$\{a^2-a^1, a^3-a^1, \dots, a^d-a^{1}\}$ are linearly independent is a set of $d^2$-dimensional Lebesgue measure $0$. 

\vskip.125in 

\subsection{Proof of part i) for $Q_d$} Let $A=\{a^1, a^2, \dots, a^{d+1}\}$. With formula (\ref{cubefourier}) in mind, note that $e^{2 \pi i x \cdot b}$ is orthogonal to each $e^{2 \pi i x \cdot a^j}$ if and only if $b$ differs from each $a^j$ by an integer in at least one coordinate. Suppose that $a^j$'s differ from one another by an irrational number in each coordinate. Then it is clear that $b$ cannot differ from each $a^j$ by an integer in at least one coordinate because there are only $d$ coordinates and there are $d+1$ $a^j$'s which differ from one another by an irrational number in each coordinate by assumption. It is also clear that the set of $d+1$-tuples $(a^1, \dots, a^{d+1})$ which do not differ from one another by an irrational number in each coordinate is a set of $d(d+1)$-dimensional Lebesgue measure $0$, so the proof is complete. 

\vskip.125in 

\subsection{Proof of part i) for $B_d$} Let $A=\{a^1, \dots, a^{d+1}\}$ and assume without loss of generality that $a^{d+1}=(0, \dots, 0)$. On the unit ball, the existence of a vector $b$ with the orthogonality property to $d+1$ vectors  $a^1,a^2, \dots, a^{d+1}$ is equivalent to say that for any $i$ , the quantity  $|a^i - b|$  belongs to the set of positive zeros of the Bessel function $J_{\frac{d}{2}}(2\pi |\xi|)$.   This also means that all $a^i$ belong to some sphere centered by a fixed vector $b$ and radii in size of  some positive zeros of  Bessel function $J_{d/2}(2\pi \cdot)$.  With this description, we can express $\mathcal A$ as 
 
 $$\mathcal A=   \bigcup_{(R_1, \cdots, R_{d+1}) \in \mathcal Z_{d+1}} \left(\bigcup_{\xi\in \Bbb R^d} S^{d-1}_{R_1}(\xi) \times \cdots \times S^{d-1}_{R_{d+1}}(\xi) \right) $$ 
where 
 $$\mathcal Z_{d+1} = 
   \underbrace{  
\mathcal Z(J_{d/2}(2\pi \cdot)) \times \cdots\mathcal Z(J_{d/2}(2\pi \cdot))}_\text{$(d+1)$-times} $$

\noindent  and $S^{d-1}_{R_i}(\xi)$ is the sphere in $\Bbb R^d$ centered at $\xi$ with radius $R_i$.
 
 Notice, the zero set $ Z(J_{d/2}(2\pi \cdot))$ of Bessel function in $\Bbb R^d$ is a countable set. Thus the set $\mathcal Z_{d+1}$  is a countable set in the product space and we can write 
 $$\mu_{d(d+1)}(\mathcal A) \leq \sum_{(R_1,\cdots,R_{d+1})\in\mathcal Z_{d+1}}\mu_{d(d+1)}\left(\bigcup_{\xi\in \Bbb R^d} S^{d-1}_{R_1}(\xi) \times \cdots \times S^{d-1}_{R_{d+1}}(\xi) \right) . $$
 
So, our problem (i.e. finding the $d(d+1)$ measure of the set $\mathcal A$) reduces to finding the measure of the set $\left(\bigcup_{\xi\in \Bbb R^d} S^{d-1}_{R_1}(\xi) \times \cdots \times S^{d-1}_{R_{d+1}}(\xi) \right) $  in $d(d+1)$ dimension.  But this can be obtained by  Fubini's Theorem along the fact that $\mu_{d(d+1)}$ is the product measure as follows: 
 
\begin{align}\notag  
\mu_{d(d+1)}\left(\bigcup_{\xi\in \Bbb R^d} S^{d-1}_{R_1}(\xi) \times \cdots \times S^{d-1}_{R_{d+1}}(\xi) \right)   &= \int_{\xi\in \Bbb R^d} 
 \mu_{d(d+1)}\left( S^{d-1}_{R_1}(\xi) \times \cdots \times S^{d-1}_{R_{d+1}}(\xi) \right) d\xi \\\notag
 &= 
 \int_{\xi\in \Bbb R^d} \prod_i    \underbrace{  \mu_d (S^{d-1}_{R_i}(\xi))}_{=0} d\xi = 0,
\end{align}
and we are done.

\subsection{Proof of part iv) in the case $\Omega=B_d$} In dimensions three and higher, consider the unit sphere centered at the origin. Place any finite number of points $a^1, a^2, \dots, a^N$ on the $(d-2)$-dimensional sphere obtained by intersecting this sphere with the plane given by the equation $x_d=0$. Every point on the line $\{t(0, \dots, 0,1): t \in {\Bbb R} \}$ is equidistant from the points $A=\{a^1, a^2 \dots, a^N\}$. Now choose a point $b$ on this line such that its distance to these points is a zero of $J_{\frac{d}{2}}(2 \pi \cdot)$. In view of the discussion in the proof of part iii), $e^{2 \pi i x \cdot b}$ is orthogonal to $e^{2 \pi i x \cdot a^j}$ for each $1 \leq j \leq N$. See Figure \ref{fig:sphere} for an illustration of the solution in dimension $d=3$. 
\begin{figure}[h!]
 \begin{tikzpicture}[scale=0.7]
\shade[ball color = gray!40, opacity = 0.4] (0,0) circle (2cm);
  \draw (0,0) circle (2cm);
  \draw (-2,0) arc (180:360:2 and 0.6)
  \foreach \p in {0,0.1,...,1} {node[pos=\p]{\huge{\bf$\cdot$}}};
  \draw[dashed] (2,0) arc (0:180:2 and 0.4)
   \foreach \p in {0,0.1,...,1} {node[pos=\p]{\huge{\bf$\cdot$}}};
   \fill[fill=black] (0,0) circle (1pt);
  \draw[dashed] (0,0 ) -- node[below]{\tiny{$r$}} (2,0);
  	\node [style=none] at (-3, 2.2) {\tiny{A zero of $J_{d/2}(2\pi .)$}};
	 \draw [style=dir, bend right] (-2.8, 1.8) to (-1.5, 1);
  \draw[dashed] (-2,0 ) -- (0,3);
      \draw[dashed] (-1.25,0.4 ) --  (0,3);
  \draw[dashed] (0,-3.5) -- (0,3.5);
  	\filldraw [black] (0, 3)  circle (2pt);
		\node [style=none]  at (0, 3) {\tiny \quad \quad $b$};
			\filldraw [black] (0, -3)  circle (2pt);
		\node [style=none]  at (0, -3) {\tiny \quad \quad $-b$};
	 \end{tikzpicture}\centering 
\caption{\footnotesize Each node on the big circle illustrates a point $a^j$ in $A$.  The distance of  the points $b$ and $-b$  on the line from each node is a zero of Bessel function.}
\label{fig:sphere}
\end{figure}
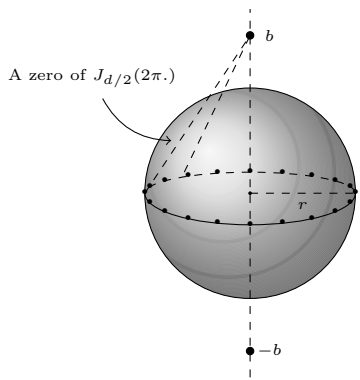
 
We use a slightly different construction in ${\Bbb R}^2$. Let $a^1=(0,0)$, $a^2$ be a vector of length $r$ in the third quadrant, and let $a^3$ be the reflection of $a^2$ across the $x_2$-axis. Let $b=(0,R)$ where $R$ is a zero of $J_1(2 \pi \cdot)$. If $R$, $r$ are sufficiently large, the formula (\ref{zeroes}) and the Intermediate Value Theorem imply that $a^2$ can be chosen such that the distance from $a^2$ to $b$ is also a zero of $J_1(2 \pi \cdot)$. By construction, the distance from $a^3$ to $b$ is the same as the distance from $a^2$ to $b$, and we see that $e^{2 \pi i x \cdot b}$ is orthogonal to $e^{2 \pi i x \cdot a^j}$, $j=1,2,3$, as needed. Indeed, this argument shows that for any $N$, we can construct $N$ $a^j$s so that $e^{2 \pi i x \cdot b}$ is orthogonal to $e^{2 \pi i x \cdot a^j}$ for all $1 \leq j \leq N$, by taking $R,r$ to be sufficiently large and repeating the same argument.     The construction is depicted in  Figure \ref{fig:pendulum}. 
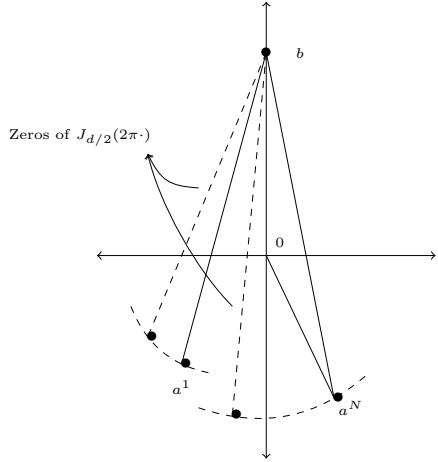
\begin{figure}[h!]
\centering
\scalebox{0.9}{\begin{tikzpicture}
	\begin{pgfonlayer}{nodelayer}
		\node [style=none] (0) at (0, 0) {};
		\node [style=none] (1) at (0, 3) { {$\bullet$} };
		\node [style=none] (2) at (-2, -0.75) {};
		\node [style=none] (3) at (-0.75, -1.75) {};
		\node [style=none] (4) at (2.5, 0) {};
		\node [style=none] (5) at (-2.5, 0) {};
		\node [style=none] (6) at (-1.75, -1.2) { { $\bullet$} };
		\node [style=none] (7) at (-1.25, -1.6) { { $\bullet$} };
		\node [style=none] (8) at (-1, -2.25) {};
		\node [style=none] (9) at (1.5, -1.75) {};
		\node [style=none] (10) at (-0.5, -2.35) { { $\bullet$} };
		\node [style=none] (11) at (1, -2.1) { { $\bullet$} };
		\node [style=none] (13) at (0, -3) {};
		\node [style=none] (14) at (0, 3.75) {};
		\node [style=none] (15) at (0.5, 3) {{\tiny $b$} };
		\node [style=none] (17) at (-1.25, -1.95) { {\tiny $a^1$}};
		\node [style=none] (18) at (1.25, -2.25) { {\tiny $a^N$}};
		\node [style=none] (19) at (-1.75, 1.5) {};
		\node [style=none] (20) at (-1, 1) {};
		\node [style=none] (21) at (-2.75, 1.75) {{\tiny Zeros of $J_{d/2}(2\pi \cdot)$}};
		\node [style=none] (22) at (-1.75, 1.5) {};
		\node [style=none] (23) at (-2, 4.5) {};
		\node [style=none] (24) at (-0.5, -0.75) {};
		\node [style=none] (25) at (0.2, 0.2) {\tiny $0$};
	\end{pgfonlayer}
	\begin{pgfonlayer}{edgelayer}
		\draw [style=dashed, bend right] (2.center) to (3.center);
		\draw [style=dashed] (1.center) to (6.center);
		\draw (1.center) to (7.center);
		\draw [style=dashed, bend right] (8.center) to (9.center);
		\draw [style=dashed] (1.center) to (10.center);
		\draw (1.center) to (11.center);
		\draw [style=dir] (0.center) to (14.center);
		\draw [style=dir] (0.center) to (13.center);
		\draw [style=dir] (0.center) to (5.center);
		\draw [style=dir] (0.center) to (4.center);
		\draw [style=dir, bend right=330, looseness=1.25] (20.center) to (19.center);
		\draw (0.center) to (11.center);
		\draw [style=dir, bend left=16, looseness=0.75] (24.center) to (22.center);
	\end{pgfonlayer}
\end{tikzpicture}}
 \caption{\footnotesize Here, $b$ is fixed and we move $a^j$ until its distance from $b$ is a zero of Bessel function.}
\label{fig:pendulum}
\end{figure}

\vskip.125in 

\subsection{Proof of part iv) in the case $\Omega=Q_d$} Let $A=\{a^1, a^2, \dots, a^N \}$, where each $a^j$ has integer coordinates. Let $b\in \Bbb R^d\setminus A$ with integer coordinates, for example. Then $e^{2 \pi i x \cdot b}$ is orthogonal to $e^{2 \pi i x \cdot a^j}$, $1 \leq j \leq N$ in view of (\ref{cubefourier}). 

\vskip.125in 

\subsection{Proof of part v)} Since $a^1 \not=a^2$, consider the bisector of these two points. Every point on the bisector is equidistant from $a^1$ and $a^2$ and every sufficiently large distance to $a^1$ and $a^2$ is realized. Therefore we can choose a point on this bisector a distance $R$ from $a^1$ and $a^2$, where $R$ is a zero of $J_1(2 \pi \cdot)$, which completes the proof in view of the discussion in the proof of part iii). 

\vskip.125in 

\subsection{Proof of part vi)} We first write down the argument in ${\Bbb R}^3$ and then indicate how to extend it to higher dimensions. Let $A=\{a^1,a^2,a^3\}$, where $a^1=(0,0,0)$, $a^2=(1,0,0)$ and $a^3=(\alpha,0,0)$. Note that the points are chosen to live on a line since otherwise the proof of part iii) would imply that ${\mathcal E}(A)$ is an incomplete system. Let $b=(b_1,b_2,b_3)$ be any point in ${\Bbb R}^3$ such that its distance to $a^1, a^2, a^3$ is $r_1, r_2, r_3$, respectively, where each $r_j$ is a zero of $J_{\frac{3}{2}}(2 \pi \cdot)$. We want to show that if $\alpha$ is chosen appropriately,  then the distances from $b$ to $a^1, a^2, a^3$ cannot in fact all be zeroes of $J_{\frac{3}{2}}$. For this, let 
\begin{subequations}\label{distances}
\begin{align}|b-a^1|^2&=  b_1^2+b_2^2+b_3^2=r_1^2, \\
  |b-a^2|^2 &=  {(b_1-1)}^2+b_2^2+b_3^2=r_2^2,  \  \ \text{and} \\
|b-a^3|^2 &=   {(b_1-\alpha)}^2+b_2^2+b_3^2=r_3^2. 
\end{align}
\end{subequations}

\noindent The relations in  (\ref{distances}) imply that 
$$ r_3^2-r_1^2-\alpha(r_2^2-r_1^2)=\alpha^2-\alpha,$$ 
or, 
$$ \alpha^2+\alpha(r_2^2-r_1^2-1)-(r_3^2-r_1^2)=0.$$ 
The quadratic formula yields 
\begin{align}\label{alpha}
\alpha=\frac{-(r_2^2-r_1^2-1) \pm \sqrt{{(r_2^2-r_1^2-1)}^2+4(r_3^2-r_1^2)}}{2}.
\end{align}

In view of (\ref{alpha}) and (\ref{zeroes}), the set of possible values of $\alpha$, as $r_1,r_2,r_3$ ranges over the zeroes of $J_{\frac{3}{2}}(2 \pi \cdot)$, is countable. It follows that there are uncountably many values of $\alpha$ for which that family ${\mathcal E}(A)$ is complete, where $A=\{a^1,a^2,a^3\}$, with $a^1=(0,0,0)$, $a^2=(1,0,0)$ and $a^3=(\alpha,0,0)$. \\

To prove the claim for $d\geq 4$, first 
 we  take 
 $$a^1=(0,0,0), a^2=(1,0,0), a^3=(\alpha_3,0,0), \dots, a^N=(\alpha_N,0,0)\in \Bbb R^3, \quad \text{with} \ N\geq 4.$$ 
 
  Repeating the calculation above yields the equations 
$$ \alpha_k^2+\alpha_k(r_2^2-r_1^2-1)-(r_k^2-r_1^2)=0, \ 3 \leq k \leq N.$$ 

Applying the quadratic formula to each equation as above shows that as $r_j$'s range over the zeroes of $J_{\frac{3}{2}}(2 \pi \cdot)$, there is only a countable number of possible values for each $\alpha_k$. This shows that $\alpha_k$'s can be chosen so that whatever $b$ we choose, $e^{2 \pi i x \cdot b}$ cannot be orthogonal to all of the $e^{2 \pi i x \cdot a^j}$'s. This argument extends readily to higher dimensions since we can view ${\Bbb R}^3$ as sitting inside ${\Bbb R}^d$, $d \ge 4$ by setting the remaining entries to equal to $0$, and then repeat the argument above for $N=d$. 

 \vskip.25in 
 
 \section{Proof of Theorem \ref{mainapprox}} 
 
 \vskip.125in 
  We shall make use of the following classical result. 
   
\begin{lemma}\label{stationaryphase} (\cite{H62}; see also \cite{GS58}) Let $K$ be a bounded symmetric convex set with boundary  $\partial K$. Given $\omega \in S^{d-1}$, let $\kappa(\omega)$ denote the Gaussian curvature of $\partial K$ at the (unique) point where the unit normal is $\omega$. Then 
\begin{equation} \label{herzformula} \widehat{\chi}_K(\xi)=\kappa^{-\frac{1}{2}} \left( \frac{\xi}{|\xi|} \right) 
\sin \left( 2 \pi \left( \rho^{*}(\xi)-\frac{d-1}{8}\right) \right){|\xi|}^{-\frac{d+1}{2}}+{\mathcal D}_K(\xi), \end{equation} where 
$$ |{\mathcal D}_K(\xi)| \leq C_K {|\xi|}^{-\frac{d+3}{2}},$$ and 
\begin{equation}\label{dualgauge} \rho^{*}(\xi)=\sup_{x \in \partial K} x \cdot \xi. \end{equation}  
\end{lemma} 
 
 To prove  Theorem \ref{mainapprox}, 
we assume, for the sake of contradiction, that $L^2(B_d)$ possesses a $\phi$-approximately orthogonal basis ${\mathcal E}(A)$ with $\phi$ satisfying (\ref{phismall}). We need the following lemma.

\begin{lemma}\label{density}   Let $\Omega$ be a domain with positive measure. 
Assume that ${\mathcal E}(A)$ is a basis for $L^2(\Omega)$. Then the following holds: 
\begin{enumerate} 
\item\label{uniformly separated}  The set $A$ is uniformly separated,  i.e., 
 $$\inf\{ |a-a'|: \  a, a'\in A, \ a\neq a'\}>0  ,$$ 
 thus, $A$ is discrete. 
 
 \vskip.125in 
 
\item\label{Beurling-density} $A$ has positive upper Beurling density, i.e., 
\begin{align}\label{upper-Beurling-Density} 
c:= \limsup_{r\to \infty} \sup_{x\in \Bbb R^d} \cfrac{\sharp(A\cap B(r,x))}{|B(r,x)|}>0.
\end{align} 
Here, $B(x,r)$ is the ball centered at $x$ with radius $r$. 
\end{enumerate} 
\end{lemma} 
\begin{proof}

To see this note that the fact that (\ref{uniformly separated}) holds is well-known (see e.g. \cite{OZ} and \cite{IK06}). To prove (\ref{Beurling-density}), let 
$$\alpha:=\inf\{ |a-a'|: \  a, a'\in A, \ a\neq a'\}.$$ By (\ref{uniformly separated}), $\alpha>0$  and it is easy to see that for any $x\in \Bbb R^d$ and $r>0$ 
  
\begin{align}\label{density_upper_estimation} \cfrac{\sharp(A \cap B(x,r))}{|B(x,r)|} \leq  \frac{(r+\alpha/2)^d}{r^d}. \end{align} Since $\alpha$ only depends on $\Omega$, the proof holds by the  relation (\ref{density_upper_estimation}). Moreover, we obtain $c\leq (\alpha/2)^d$. 
     \end{proof} 

Let  $0<\delta<\alpha$ and  let $E_{\delta}$ denote the $\delta$-neighborhood of set $A$, as  depicted in Figure \ref{fig:dots}.   
\begin{center}
\begin{figure}[h!]
\scalebox{0.8}{\begin{tikzpicture}
	\begin{pgfonlayer}{nodelayer}
		\node [style=none] (0) at (-5, 4) {};
		\node [style=none] (1) at (-5, 0) {};
		\node [style=none] (2) at (5, 0) {};
		\node [style=none] (3) at (5, 4) {};
		\node [circle, fill, scale=0.5, draw=black] (4) at (-4.5, 3.75) {};
		\node [circle, fill, scale=0.5, draw=black] (5) at (-4.75, 3.25) {};
		\node [circle, fill, scale=0.5, draw=black] ((6) at (-3.75, 3.75) {};
		\node [circle, fill, scale=0.5, draw=black] ( (7) at (-4, 3.25) {};
		\node [circle, fill, scale=0.5, draw=black] ( (8) at (-4.5, 2.5) {};
		\node [circle, fill, scale=0.5, draw=black] ( (9) at (-3.75, 2.75) {};
		\node [circle, fill, scale=0.5, draw=black] ((10) at (-3, 3.25) {};
		\node [circle, fill, scale=0.5, draw=black] ((11) at (-2, 3.75) {};
		\node [circle, fill, scale=0.5, draw=black] (12) at (-2.25, 3.25) {};
		\node [circle, fill, scale=0.5, draw=black] (13) at (-2.75, 2.5) {};
		\node [circle, fill, scale=0.5, draw=black] (14) at (-3.75, 2) {};
		\node [circle, fill, scale=0.5, draw=black] (15) at (-4.5, 1.75) {};
		\node [circle, fill, scale=0.5, draw=black] (16) at (-4.75, 1) {};
		\node [circle, fill, scale=0.5, draw=black] (17) at (-3.5, 1) {};
		\node [circle, fill, scale=0.5, draw=black] (18) at (-4.25, 1.25) {};
		\node [circle, fill, scale=0.5, draw=black] (19) at (-2.75, 1.5) {};
		\node [circle, fill, scale=0.5, draw=black] (20) at (-4.25, 0.5) {};
		\node [circle, fill, scale=0.5, draw=black] (21) at (-3, 0.5) {};
		\node [circle, fill, scale=0.5, draw=black] (22) at (-1.25, 3.5) {};
		\node [circle, fill, scale=0.5, draw=black] (23) at (-1.5, 3) {};
		\node [circle, fill, scale=0.5, draw=black] (24) at (-2, 2.5) {};
		\node [circle, fill, scale=0.5, draw=black] (25) at (-0.25, 3.5) {};
		\node [circle, fill, scale=0.5, draw=black] (26) at (-0.75, 2.5) {};
		\node [circle, fill, scale=0.5, draw=black] (27) at (-1.5, 2) {};
		\node [circle, fill, scale=0.5, draw=black] (28) at (-2, 1.5) {};
		\node [circle, fill, scale=0.5, draw=black] (29) at (-0.5, 1.75) {};
		\node [circle, fill, scale=0.5, draw=black] (30) at (0.25, 2.75) {};
		\node [circle, fill, scale=0.5, draw=black] (31) at (0.5, 3.5) {};
		\node [circle, fill, scale=0.5, draw=black] (32) at (1.5, 3.5) {};
		\node [circle, fill, scale=0.5, draw=black] (33) at (1.25, 3) {};
		\node [circle, fill, scale=0.5, draw=black] (34) at (1, 2.25) {};
		\node [circle, fill, scale=0.5, draw=black] (35) at (0.25, 2.25) {};
		\node [circle, fill, scale=0.5, draw=black] (36) at (0, 1) {};
		\node [circle, fill, scale=0.5, draw=black] (37) at (0.5, 1.75) {};
		\node [circle, fill, scale=0.5, draw=black] (38) at (-1.25, 1.25) {};
		\node [circle, fill, scale=0.5, draw=black] (39) at (-0.75, 1) {};
		\node [circle, fill, scale=0.5, draw=black] (40) at (-2.25, 0.75) {};
		\node [circle, fill, scale=0.5, draw=black] (41) at (-1.5, 0.75) {};
		\node [circle, fill, scale=0.5, draw=black] (42) at (-0.75, 0.25) {};
		\node [circle, fill, scale=0.5, draw=black] (43) at (-2.25, 0.25) {};
		\node [circle, fill, scale=0.5, draw=black] (44) at (0, 0.25) {};
		\node [circle, fill, scale=0.5, draw=black] (45) at (0.75, 1) {};
		\node [circle, fill, scale=0.5, draw=black] (46) at (1.25, 0.75) {};
		\node [circle, fill, scale=0.5, draw=black] (47) at (1.5, 1.75) {};
		\node [circle, fill, scale=0.5, draw=black] (48) at (0.75, 0.5) {};
		\node [circle, fill, scale=0.5, draw=black] (49) at (2.25, 0.5) {};
		\node [circle, fill, scale=0.5, draw=black] (50) at (2.25, 1) {};
		\node [circle, fill, scale=0.5, draw=black] (51) at (1.5, 0.25) {};
		\node [circle, fill, scale=0.5, draw=black] (52) at (2.5, 2) {};
		\node [circle, fill, scale=0.5, draw=black] (53) at (2.75, 3) {};
		\node [circle, fill, scale=0.5, draw=black] (54) at (1.75, 2.5) {};
		\node [circle, fill, scale=0.5, draw=black] (55) at (2, 3.5) {};
		\node [circle, fill, scale=0.5, draw=black] (56) at (3, 3.5) {};
		\node [circle, fill, scale=0.5, draw=black] (57) at (4.5, 3.5) {};
		\node [circle, fill, scale=0.5, draw=black] (58) at (3.75, 3.75) {};
		\node [circle, fill, scale=0.5, draw=black] (59) at (3.75, 3) {};
		\node [circle, fill, scale=0.5, draw=black] (60) at (4.5, 3) {};
		\node [circle, fill, scale=0.5, draw=black] (61) at (4.5, 2) {};
		\node [circle, fill, scale=0.5, draw=black] (62) at (4, 2.5) {};
		\node [circle, fill, scale=0.5, draw=black] (63) at (3.25, 2.25) {};
		\node [circle, fill, scale=0.5, draw=black] (64) at (3.5, 1.75) {};
		\node [circle, fill, scale=0.5, draw=black] (65) at (4, 1.5) {};
		\node [circle, fill, scale=0.5, draw=black] (66) at (3.25, 0.75) {};
		\node [circle, fill, scale=0.5, draw=black] (67) at (3.25, 1.25) {};
		\node [circle, fill, scale=0.5, draw=black] (68) at (4.25, 0.75) {};
		\node [circle, fill, scale=0.5, draw=black] (69) at (4.5, 1.25) {};
		\node [circle, fill, scale=0.5, draw=black] (70) at (3.75, 0.5) {};
		\node [circle, fill, scale=0.5, draw=black] (71) at (2.75, 1.5) {};
		\node [circle, fill, scale=0.5, draw=black] (72) at (2, 1.5) {};
		\node [circle, fill, scale=0.5, draw=black] (73) at (2.25, 2.75) {};
		\node [circle, fill, scale=0.5, draw=black] (74) at (4.75, 0.25) {};
		\node [style=none] (75) at (6, 3.25) {};
		\node [style=none] (76) at (6.8, 3.25) {{\small $B_d(a,\delta)$}};
		\node [style=none] (77) at (6, 3.75) {};
		\node [style=none] (78) at (5.5, 3.5) {};
		\node [style=none] (80) at (0, -0.5) {{\small $\Bbb R^d$}};

	\end{pgfonlayer}
	\begin{pgfonlayer}{edgelayer}
	\filldraw[fill=gray!50!white, draw=black]  (-5, 4)  -- (-5, 0) -- (5, 0) -- (5, 4) --  (-5, 4) ;
		\draw [style=dir, dashed, bend left]  (75.center) to  (4.7, 2) ;
		\draw [style=dir, dashed, bend left]  (75.center) to (4.7, 3.4);
	\end{pgfonlayer}
\end{tikzpicture}}
\caption{\footnotesize In this graph, each node represents a ball centered at $a\in A$ with radius $\delta$.}
\label{fig:dots}
\end{figure}
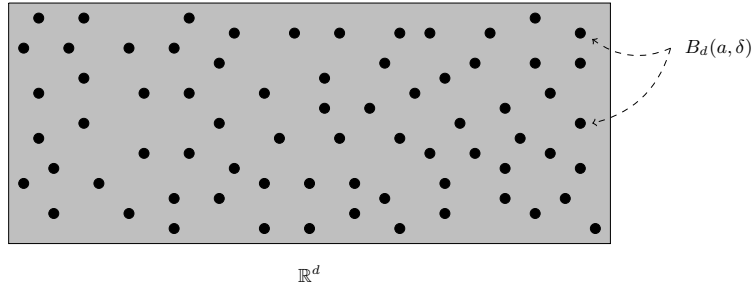
\end{center}
This allows us to obtain an estimate on the upper and lower Lebesgue density of the thickened set, as follows: 
\begin{lemma} \label{furstenbergdensity} $E_\delta$ has positive  upper and lower Lebesgue density. 
\end{lemma} 

\vskip.125in 
The proof of the lemma  is an immediate result of  Lemma \ref{density} (\ref{Beurling-density}). Indeed, for $0<\delta < \alpha$, the set $E_\delta$ is a disjoint union of  all  the  balls with centers in    $A$ and radius $\delta$.  Therefore we have 
\begin{align}\label{density_upper_estimation_delta} |E_\delta \cap B(x,r)| =  \sum_{\{a\in A:  |a-x|<r+\delta\}} |B_d(a,\delta)\cap B(x,r)|  \leq \delta^d \sharp(A\cap B(x,r+\delta)), \end{align}
hence the assertion of the lemma holds immediately. 

\vskip.125in 
 
The remainder of the proof of Theorem \ref{mainapprox} is going to use Lemma \ref{furstenbergdensity} and the following result due to Furstenberg, Katznelson and Weiss. 

\begin{theorem} (Furstenberg, Katznelson and Weiss (1986)) (\cite{FKW90}, Theorem A) \label{FKWtheorem} Let $E \subset {\Bbb R}^d$ be a set of positive upper Lebesgue density, in the sense that $ \limsup_{r \to \infty} \sup_x\frac{|E \cap B(x,r)|}{|B(x,r)|}=c>0$. Then there exists a threshold $L_0(E)$ such that for all $L>L_0$, there exist $x,y \in E$ such that $|x-y|=L$. In other words, every sufficiently large distance is realized in $E$. \end{theorem} 
  
 Applying 
   Theorem \ref{FKWtheorem} to the set $E_\delta$, it follows that there exists $L_0>0$ such that for every $L >L_0$, there exist $x,x' \in E_{\delta}$ with 
   distance exactley $L$, i.e, $$|x-x'|=L.$$

When  $K=B_d$ is the unit ball,  we have $\rho^{*}(\xi)=|\xi|$ in (\ref{dualgauge}). Applying Lemma \ref{stationaryphase} and the condition (\ref{phismall}), we see that for given $\epsilon>0$ there exists $R>0$ such that if $a,a' \in A$  with $|a-a'|>R$, then 
$$ \left|\sin \left(2 \pi \left(|a-a'|-\frac{d-1}{8} \right)\right)\right|<\epsilon.$$ 

\vskip.125in 

\noindent It follows that for some positive integer $k$ we must have 
\begin{align}\label{distance}
 |a-a'|=\frac{k}{2}+\frac{d-1}{8}+O(\epsilon) . 
 \end{align}

Let $E_{\delta}$ be as above.  Then for every $\epsilon>0$ there exists $R>0$ such that if $|x-x'|>R$, $x,x' \in E_{\delta}$, then by (\ref{distance}) we have 
\begin{equation} \label{busted} |x-x'|=\frac{k}{2}+\frac{d-1}{8}+O(\epsilon)+O(\delta). \end{equation}

If $\epsilon$ and $\delta$ are taken to be sufficiently small, then any sufficient large distance $L$ can not be realized in the set $E_\delta$  due to the relation  (\ref{busted}). So, we arrive to the contradiction by  Theorem \ref{FKWtheorem}, hence the Theorem \ref{mainapprox} is proved.


\vskip.125in 

\begin{thebibliography}{20}

\bibitem{PhilippBirklbauer} P. Birklbauer, {\it Fuglede Conjecture holds in $\Bbb Z_5^3$}, Experimental Mathematics, published online: 12 Jul (2019), 
DOI:  {\tt https://doi.org/10.1080/10586458.2019.1636427}.


\bibitem{FMV2019} T. Fallon, A. Mayeli,  D.Villano, {\it The Fuglede Conjecture holds in \(\Bbb F_p^3\) for p=5,7}, to appear in Proceeding of AMS, DOI: {\tt https://doi.org/10.1090/proc/14750}.

\bibitem{Fug01} B. Fuglede, {\it Orthogonal Exponentials on the Ball}, Expo. Math. \textbf{19}, (2001), 267-272. 
 
\bibitem{Fug74} B. Fuglede, {\it Commuting self-adjoint partial differential operators and a group theoretic problem}, J. Funct. Anal. \textbf{16} (1974), 101-121.

\bibitem{FKW90} H. Furstenberg, Y. Katznelson, and B. Weiss, {\it Ergodic theory and configurations in sets of positive density} Mathematics of Ramsey theory, 184-198, Algorithms Combin., 5, Springer, Berlin, (1990).

\bibitem{GS58} I. Gelfand and G. Shilov, {\it Generalized Functions}, Vol. 1, Academic Press, (1958).

\bibitem{GL17} R. Greenfeld and N. Lev, {\it Fuglede's spectral set conjecture for convex polytopes}, Anal. PDE 10 (2017), no. 6, 1497-1538.

\bibitem{GL20} R. Greenfeld and N. Lev, {\it Spectrality of product domains and Fuglede's conjecture for convex polytopes}, J. Anal. Math. 140 (2020), no. 2, 409-441. 

\bibitem{H62} C. Herz, {\it Fourier transforms related to convex sets}, Ann. of Math. (2) \textbf{75} (1962) 81-92. 

\bibitem{IK06} A. Iosevich and M. Kolountzakis, {\it A Weyl type formula for Fourier spectra and frames}, Proc. Amer. Math. Soc. \textbf{134} (2006), no. 11, 3267-3274.

\bibitem{IK13} A. Iosevich and M. Kolountzakis, {\it Size of orthogonal sets of exponentials for the disk}, Rev. Mat. Iberoam. \textbf{29} (2013), no. 2, 739-747.

\bibitem{IKT03} A. Iosevich, N. Katz and T. Tao, {\it The Fuglede spectral conjecture holds for convex planar domains}, Math. Res. Lett. \textbf{10} (2003), no. 5-6, 559-569.

\bibitem{IKP01} A. Iosevich, N.~H. Katz and S. Pedersen, {\it Fourier bases and a distance problem of Erd\"os}. \newblock {\em Math. Res. Lett.}, \textbf{6}, 251-255, (2001).

\bibitem{IKT01} A. Iosevich, N.~H. Katz and T. Tao, {\it Convex bodies with a point of curvature do not have Fourier bases}, Amer. J. Math., \textbf{123}, 115-120, (2001).

\bibitem{IMP17} A. Iosevich, A. Mayeli and J. Pakianathan, {\it The Fuglede Conjecture holds in $\Bbb Z_p\times \Bbb Z_p$}, Anal. PDE 10 (2017), no. 4, 757-764.

\bibitem{IR03} A. Iosevich and M. Rudnev, {\it A combinatorial approach to orthogonal exponentials}, Int. Math. Res. Not. (2003), no. 50, 2671-2685. 

\bibitem{KM06} M. Kolountzakis and M. Matolcsi, {\it Tiles with no spectra}, Forum Math. \textbf{18} (2006), no. 3, 519-528. 

\bibitem{K99} M. Kolountzakis, {\it Non-symmetric convex domains have no basis of exponentials}, Illinois journal of mathematics 44(3),  March 1999

\bibitem{LM17} C.K. Lai and A. Mayeli, {\it Non-separable lattices, Gabor orthonormal bases and tilings}, J. Fourier Anal. Appl. \textbf{25} (2019), no. 6, 3075-3103.

\bibitem{LevM19} N. Lev and M. Matolcsi, {\it The Fuglede conjecture for convex domains holds in all dimensions}, (arXiv:1904.12262), (2019). 

\bibitem{M05} M. Matolcsi, {\it Fuglede conjecture fails in dimension $4$}, Proc. Amer. Math. Soc., \textbf{133} (2005), no. 10, 3021-3026.

\bibitem{OZ} T. E. Olson and R. A. Zalik, {\it Nonexistence of a Riesz basis of translates}, Approximation Theory, Lecture Notes in Pure and Applied Math., {\bf 138}, Dekker, New York (1992) 401--408. MR1174120

\bibitem{T04} T. Tao, {\it Fuglede's conjecture is false in 5 and higher dimensions}, Math. Res. Lett. \textbf{11} (2004), no. 2-3, 251-258.

\bibitem{Z84} V. P. Zastavnyi, {\it On the set of zeros of the Fourier transform of the measure and summation of double Fourier series by Bernshtein-Rogozinski type methods}, [in Russian], Ukr. Mat. Zhurn. \textbf{36}, No. 5, 615-621 (1984).
\end{thebibliography}
\end{document}